\documentclass[oneside]{article}
\usepackage{amsmath,amssymb,amscd,amsthm,verbatim,alltt,amsfonts,array}
\usepackage{mathrsfs}
\usepackage[english]{babel}
\usepackage{latexsym}
\usepackage{amssymb}
\usepackage{euscript}
\usepackage{graphicx}
\usepackage{arydshln}
\usepackage[dvipsnames]{xcolor}

\newtheorem{theorem}{Theorem}
\theoremstyle{plain}

\newtheorem{conjecture}{Conjecture}
\newtheorem{corollary}{Corollary}

\newtheorem{remark}{Remark}

\numberwithin{equation}{section}

\begin{document}

{\footnotesize%
\hfill
}

  \vskip 1.2 true cm

\begin{center} {\bf On the quasi polynomiality of extremal homology of configuration spaces} \\
          {by}\\
{\sc Muhammad Yameen}
\end{center}

\pagestyle{myheadings}
\markboth{On the quasi polynomiality of extremal homology of configuration spaces}{Muhammad Yameen}

\begin{abstract}
Consider the unordered configuration spaces of manifolds. Knudsen, Miller and Tosteson proved that the extremal homology groups of configuration spaces of manifold are eventually quasi polynomials. In this paper, we give the precise degree of top non-trivial quasi polynomials. This shows that the upper bound of Knudsen, Miller and Tosteson for the degree of quasi polynomials is sharp for every manifold. 
\end{abstract}

\begin{quotation}

\end{quotation}

\thispagestyle{empty}

\section{Introduction}

\label{sec:intro}


For any manifold $M$, let
$$F_{n}(M):=\{(x_{1},\ldots,x_{n})\in M^{n}| x_{i}\neq x_{j}\,for\,i\neq j\}$$
be the configuration space of $n$ distinct ordered points in $M$ with induced topology. The symmetric group $S_{k}$ acts on $F_{k}(M)$ by permuting the coordinates. The quotient $$B_{n}(M):=F_{n}(M)/S_{n} $$
is the unordered configuration space with quotient topology. 

It is well known fact that the homology groups $H_{i}(B_{n}(M);\mathbb{Q})$ are vanish for $i\geq\nu_{n},$ where $\nu_{n}=(d-1)n+1.$ In the paper \cite{KMT} (see also \cite{Y}), Knudsen, Miller and Tosteson prove that the extremal homology groups of configuration spaces of manifold are eventually quasi polynomials:

\begin{theorem}\label{Th1}
Let $M$ be a manifold of even dimension $d\geq 2.$ For each $i\geq 0,$ there is a quasi-polynomial in $n$ of degree at most dim $H_{d-1}(M;\mathbb{Q}^{w})-1)$ and period at most 2, which coincides with dim $H_{\nu_{n}-i}(B_{n}(M);\mathbb{Q})$ for all $n$ sufficiently large.
\end{theorem}
If $H_{d-1}(M;\mathbb{Q})=0,$ then the extremal homology groups $H_{\nu_{n}-i}(B_{n}(M);\mathbb{Q})$ are eventually vanish. Equivalently, Theorem \ref{Th1} states that there are two polynomials $p_{M}^{\nu_{n}-i}(n)$ and $q_{M}^{\nu_{n}-i}(n)$ such that 
$$ Q_{M}^{\nu_{n}-i}(n) =\begin{cases}
      p_{M}^{\nu_{n}-i}(n) & \mbox{ $n$ is even}\\
      q_{M}^{\nu_{n}-i}(n) & \mbox{ $n$ is odd}.\\
   \end{cases} $$
where $ Q_{M}^{\nu_{n}-i}(n)=\mbox{dim} H_{\nu_{n}-i}(B_{n}(M);\mathbb{Q}).$ The degree of quasi-polynomial $ Q_{M}^{\nu_{n}-i}(n)$ is maximum of $ deg(p_{M}^{\nu_{n}-i}(n))$ and $deg(q_{M}^{\nu_{n}-i}(n).$ We give the precise degree of top non-trivial quasi polynomials for every orientable manifold. 
\begin{theorem}\label{Th2}
Let $M$ be a closed orientable manifold of even dimension $d\geq 2.$ If $H_{d-1}(M;\mathbb{Q})$ is non-trivial then the degree of quasi-polynomial $ Q_{M}^{\nu_{n}}(n)$ is dim$H_{d-1}(M;\mathbb{Q})-1.$ 
\end{theorem}
If $M$ is not closed then the homology group $H_{\nu_{n}}(B_{n}(M);\mathbb{Q})$ is vanish. In this case, we will focus on the homology group $H_{\nu_{n}-1}(B_{n}(M);\mathbb{Q}).$ 
\begin{theorem}\label{Th3}
Let $M$ be an orientable manifold of even dimension $d\geq 2.$ If $M$ is not closed and $H_{d-1}(M;\mathbb{Q})$ is non-trivial then the degree of quasi-polynomial $ Q_{M}^{\nu_{n}-1}(n)$ is dim$H_{d-1}(M;\mathbb{Q})-1.$ 
\end{theorem}
In light of main theorems, we formulate the following conjecture:
\begin{conjecture}\label{Con}
Let $M$ be an orientable manifold of even dimension $d\geq 2.$ If $H_{d-1}(M;\mathbb{Q})$ and $ Q_{M}^{\nu_{n}-i}(n)$ are non-trivial then the degree of quasi-polynomial $ Q_{M}^{\nu_{n}-i}(n)$ is dim$H_{d-1}(M;\mathbb{Q})-1$ for $i\geq0.$ 
\end{conjecture}
\begin{remark}
Drummond-Cole and Knudsen \cite{D-K} computed the all Betti numbers of configuration spaces of surfaces. From their computations, we see that the Conjecture \ref{Con} is true for surfaces.
\end{remark}

\subsection*{Notations}

$\bullet$ We work throughout with finite dimensional graded vector spaces. The degree of an element $v$ is written $|v|$.\\\\
$\bullet$ The symmetric algebra $Sym(V^{*})$ is the tensor product of a polynomial algebra and an exterior algebra:
$$ Sym(V^{*})=\bigoplus_{k\geq0}Sym^{k}(V^{*})=Poly(V^{even})\bigotimes Ext(V^{odd}), $$
where $Sym^{k}$ is generated by the monomials of length $k.$\\\\
$\bullet$ The $n$-th suspension of the graded vector space $V$ is the graded vector space $V[n]$ with
$V[n]_{i} = V_{i-n},$ and the element of $V[n]$ corresponding to $a\in V$ is denoted $s^{n}a;$ for example
$$ H^{*}(S^{2};\mathbb{Q})[n] =\begin{cases}
      \mathbb{Q}, & \text{if $*\in\{n,n+2 \}$} \\
      0, & \mbox{otherwise}.\\
   \end{cases} $$ \\\\
$\bullet$ We write $H_{-*}(M;\mathbb{Q})$ for the graded vector space whose degree $-i$ part is
the $i$-th homology group of $M;$ for example
$$ H^{-*}(\text{CP}^m;\mathbb{Q}) =\begin{cases}
      \mathbb{Q}, & \text{if $*\in\{-2m,-2m+2,\ldots,0. \}$} \\
      0, & \mbox{otherwise}.\\
   \end{cases} $$


\section{Chevalley–Eilenberg complex}
F\'{e}lix--Thomas \cite{F-Th} (see also \cite{F-Ta}) constructed the model for rational cohomology of unordered configuration spaces of closed orientable even dimension manifolds. More recently, the identification was established in full generality by the Knudsen in \cite{Kn}. We will restrict our attention to the case
of orientable even dimensional manifolds.

Let us introduced some notations. Consider two graded vector spaces $$V^{*}=H^{-*}_{c}(M;\mathbb{Q})[d],\quad W^{*}=H_{c}^{-*}(M;\mathbb{Q})[2d-1]:$$
where
$$ V^{*}=\bigoplus_{i=0}^{d}V^{i},\quad W^{*}=\bigoplus_{j=d-1}^{2d-1}W^{j}.$$
We choose bases in $V^{i}$ and $W^{j}$ as 
$$V^{i}=\mathbb{Q}\langle v_{i,1},v_{i,2},\ldots\rangle,\quad W^{j}=\mathbb{Q}\langle w_{j,1},w_{j,2},\ldots\rangle$$
(the degree of an element is marked by the first lower index; $x_{i}^{l}$ stands for the product $x_{i}\wedge x_{i}\wedge\ldots\wedge x_{i}$ of $l$-factors). Always we take $V^{0}=\mathbb{Q}\langle v_{0}\rangle$. Now consider the graded algebra
$$ \Omega^{*,*}_{n}(M)=\bigoplus_{i\geq 0}\bigoplus_{\omega=0}^{\left\lfloor\frac{n}{2}\right\rfloor}
\Omega^{i,\omega}_{n}(M)=\bigoplus_{\omega=0}^{\left\lfloor\frac{n}{2}\right\rfloor}\,(Sym^{n-2\omega}(V^{*})\otimes Sym^{\omega}(W^{*})) $$
where the total degree $i$ is given by the grading of $V^{*}$ and $W^{*}$. We called $\omega$ is a weight grading. The differential $\partial:Sym^{2}(V^{*})\rightarrow W^{*}$ is defined as a coderivation by the equation 
$$\partial(s^{d}a\wedge s^{d}b)=(-1)^{(d-1)|b|}s^{2d-1}(a\cup b),$$ where $$\cup\,:H^{-*}_{c}(M;\mathbb{Q})^{\otimes2}\rightarrow H^{-*}_{c}(M;\mathbb{Q})$$
(here $H^{-*}_{c}$ denotes compactly supported cohomology of $M$). The degree of $\partial$ is $-1.$ It can be easily seen that $s^{d}a,\,s^{d}b\in V^{*}$ and $s^{2d-1}(a\cup b)\in W^{*
}.$ The differential $\partial$ extends over  $\Omega^{*,*}_{n}(M)$ by co-Leibniz rule. By definition the elements in $V^{*}$ have length 1 and weight 0 and the elements in $W^{*}$ have length 2 and weight 1. By definition of differential, we have 
$$\partial:\Omega^{*,*}_{n}(M)\longrightarrow\Omega^{*-1,*+1}_{n}(M).$$

\begin{theorem}
	If $d$ is even, $H_{*}(B_{n}(M);\mathbb{Q})$  is isomorphic to the homology of the complex 
	$$ (\Omega^{*,*}_{n}(M),\partial).$$
\end{theorem}
For a closed manifold, the compactly supported cohomology is the ordinary cohomology. In this case the two graded vector spaces are $$V^{*}=H_{-*}(M;\mathbb{Q})[d],\quad W_{*}=H_{-*}(M;\mathbb{Q})[2d-1].$$
Now, we will define the dual complex of $(\Omega^{*,*}_{n}(M),\partial).$ First, we define a dual differential $D$ on $\Omega^{*,*}_{n}(M).$ The dual differential is defined as $$D|_{V^{*}}=0,\quad D|_{W^{*}}:\,W^{*}\simeq H_{*}(M;\mathbb{Q})\xrightarrow{\Delta} Sym^{2}(V^{*})\simeq  Sym^{2}(H_{*}(M;\mathbb{Q})),$$
where
$\Delta$ is a diagonal comultiplication corresponding to cup product. By definition of differential, we have 
$$D:\Omega^{*,*}_{n}(M)\longrightarrow\Omega^{*+1,*-1}_{n}(M).$$
\begin{theorem}
	If $d$ is even and $M$ is closed, then $H^{*}(B_{n}(M);\mathbb{Q})$  is isomorphic to the cohomology of the complex 
	$$ (\Omega^{*,*}_{n}(M),D).$$
\end{theorem}

\section{Reduced Chevalley–Eilenberg complex}
In this section, we define an acyclic subcomplex of $(\Omega^{*,*}_{n}(M),D).$ 
\begin{theorem}\label{Acyclic}
Let $M$ be a closed orientable manifold of dimension $d.$ The subspace $$\Omega_{n-2}^{*,*}(M).(v_{d}^{2}, w_{2d-1})< \Omega_{n}^{*,*}(M)$$ is acyclic for $n\geq 2.$
\end{theorem}
\begin{proof}
Let $M$ is closed and orientable. An element in $\Omega_{n-2}^{*,*}(M).(v_{d}^{2}, w_{2d-1})$ has a unique expansion $v_{d}^{2}A+Bw_{2d-1},$ where $A$ and $B$ have no monomial containing $w_{2d-1}.$ The operator $$h(v_{d}^{2}A+Bw_{2d-1})=Bv_{d}^{2}$$ gives a homotopy $id\simeq 0.$

\end{proof}
We denote the reduced complex $(\Omega_{n}^{*,*}(M)/\Omega_{n-2}^{*,*}(M).(v_{d}^{2}, w_{2d-1}),D_{\text{induced}})$ by 
$$({}^{r}\Omega_{n}^{*,*}(M),D).$$
\begin{corollary}\label{reduced}
If $n\geq2$ and $M$ is closed orientable, then we have an isomorphism $$H^{*}({}^{r}\Omega_{n}^{*,*}(M),D)\cong H^{*}(B_{n}(M)).$$
\end{corollary}
\begin{remark}
If $M$ is not closed then the subspace $\Omega_{n-2}^{*,*}(M).(v_{d}^{2}, w_{2d-1})$ is vanish. 
\end{remark}

\section{Proof of Theorem \ref{Th2}}
In this section, we give the proof of Theorem \ref{Th2}. \\\\
\textit{Proof of Theorem \ref{Th2}.}
Let $M$ be a closed orientable. Assume $H_{d-1}(M;\mathbb{Q})$ is non-trivial and $H_{d-1}(M;\mathbb{Q})=k.$ The corresponding two vector spaces are following 
$$V^{*}=\oplus_{i=0}^{d}V^{i},\quad W^{*}=\oplus_{i=d-1}^{2d-1}W^{i}$$
where 
$$W^{2d-2}=\langle w_{2d-2,1},\ldots, w_{2d-2,k}\rangle,\quad V^{d-1}=\langle v_{d-1,1}\ldots,v_{d-1,k-1}\rangle.$$
There is no element of degree greater than $\nu_{n}$ in reduced complex. Therefore, we just focus on the degree $\nu_{n}.$ We will use the notation $\mbox{J}=\langle w_{2d-2,1},\ldots, w_{2d-2,k}\rangle.$
Let $n$ be odd. We have 
$${}^{r}\Omega^{\nu_{n},\lfloor\frac{n}{2}\rfloor}_{n}(M)=v_{d}\mbox{J}^{\lfloor\frac{n}{2}\rfloor}.$$
The cardinality of the bases elements of $\mbox{J}^{l}$ is $\binom{l+k-1}{k-1}.$ The cardinality of the bases elements of ${}^{r}\Omega^{\nu_{n},\lfloor\frac{n}{2}\rfloor}_{n}(M)$ is $\binom{\frac{n-2}{2}+k-1}{k-1}.$ Moreover, the differential of each bases element of $\mbox{J}$ is $$D(w_{2d-2,i})=v_{d}v_{d-1,i}.$$  We have $$D(v_{d}w_{2d-2,i})=0,\quad\mbox{for }i\in\{1,\ldots,k-1\}.$$ Note that $v_{d}^{\geq2}=0$ in reduced complex. Also, ${}^{r}\Omega^{*,j>\lfloor\frac{n}{2}\rfloor}_{n}(M)=0.$ The differential has bi-degree $(1,-1).$ Therefore each $v_{d}w_{2d-2,i}$ gives a cohomology class. We can write 
$$q^{\nu_{n}}_{M}(n)=\binom{\frac{n-2}{2}+k-1}{k-1}+\overline{q}_{M}^{\nu_{n}}(n).$$
where $\overline{q}_{M}^{\nu_{n}}(n)$ is a polynomial in $n.$ We can write 
$$\binom{\frac{n-2}{2}+k-1}{k-1}=\dfrac{(\frac{n-2}{2}+1)(\frac{n-2}{2}+2))\ldots(\frac{n-2}{2}+k-1)}{(k-1)!}.$$

This implies that the degree of $q^{\nu_{n}}_{M}(n)$ is at least $k-1.$ From Theorem \ref{Th1}, the degree of quasi-polynomial $Q_{M}^{\nu_{n}}$ is at most $k-1.$ Hence the degree of $Q_{M}^{\nu_{n}}$ is $k-1.$ $\hfill \square$\\

\section{Proof of Theorem \ref{Th3}}
In this section, we give the proof of Theorem \ref{Th3}. \\\\
\textit{Proof of Theorem \ref{Th3}.}  Let $M$ be a orientable manifold. Assume $M$ is not closed and $H_{d-1}(M;\mathbb{Q})$ is non-trivial. Furthermore, $H_{d-1}(M;\mathbb{Q})=k.$ The corresponding two vector spaces are following 
$$V^{*}=\oplus_{i=0}^{d}V^{i},\quad W^{*}=\oplus_{i=d-1}^{2d-1}W^{i}$$
where 
$$W^{2d-2}=\langle w_{2d-2,1},\ldots, w_{2d-2,k}\rangle,\quad V^{d-1}=\langle v_{d-1,1}\ldots,v_{d-1,k-1}\rangle.$$
There is no element of degree greater than $\nu_{n}-1$ in complex. Therefore, we just focus on the degree $\nu_{n}-1.$ We will use the notation 
$$\mbox{I}=\langle v_{d-1,1},\ldots, v_{d-1,k}\rangle, \quad\mbox{J}=\langle w_{2d-2,1},\ldots, w_{2d-2,k}\rangle.$$
Let $n$ be odd. We have 
$${}^{r}\Omega^{\nu_{n}-1,\lfloor\frac{n}{2}\rfloor}_{n}(M)=\mbox{I}\mbox{J}^{\lfloor\frac{n}{2}\rfloor}.$$
 The cardinality of the bases elements of ${}^{r}\Omega^{\nu_{n}-1,\lfloor\frac{n}{2}\rfloor}_{n}(M)$ is $k\binom{\frac{n-2}{2}+k-1}{k-1}.$  We have $$\partial(v_{d-1,j}w_{2d-2,i})=0,\quad\mbox{for }i,j\in\{1,\ldots,k-1\}.$$ The differential has bi-degree $(-1,1)$ and ${}^{r}\Omega^{j\geq\nu_{n},*}_{n}(M)=0.$ Therefore each $v_{d-1,j}w_{2d-2,i}$ gives a homology class. We can write 
$$q^{\nu_{n}-1}_{M}(n)=k\binom{\frac{n-2}{2}+k-1}{k-1}+\overline{q}_{M}^{\nu_{n}-1}(n).$$
where $\overline{q}_{M}^{\nu_{n}-1}(n)$ is a polynomial in $n.$ We can write 
$$\binom{\frac{n-2}{2}+k-1}{k-1}=\dfrac{(\frac{n-2}{2}+1)(\frac{n-2}{2}+2))\ldots(\frac{n-2}{2}+k-1)}{(k-1)!}.$$

This implies that the degree of $q^{\nu_{n}-1}_{M}(n)$ is at least $k-1.$ From Theorem \ref{Th1}, the degree of quasi-polynomial $Q_{M}^{\nu_{n}-1}$ is at most $k-1.$ Hence the degree of $Q_{M}^{\nu_{n}-1}$ is $k-1.$ $\hfill \square$\\

\noindent\textbf{Acknowledgement}\textit{.} The author gratefully acknowledge the support from the ASSMS, GC university Lahore. This research is partially supported by Higher Education Commission of Pakistan.

\vskip 0,65 true cm




\vskip 0,65 true cm







\null\hfill  Abdus Salam School of Mathematical Sciences,\\
\null\hfill  GC University Lahore, Pakistan. \\
\null\hfill E-mail: {yameen99khan@gmail.com}

\end{document}